\documentclass{article}       
\usepackage{amsthm}
\usepackage{graphicx}
\usepackage{amsmath}
\usepackage{amssymb}
\usepackage{multirow}
\usepackage{tabularx}
\usepackage{booktabs}
\usepackage{enumerate}
\usepackage{caption}
\usepackage{authblk}

\newtheorem{theorem}{Theorem}

\newtheorem{lemma}{Lemma}

\newtheorem{proposition}{Proposition}

\begin{document}

\title{A family of spectral gradient methods for optimization
}


\author{Yu-Hong Dai \thanks{LSEC, ICMSEC, Academy of Mathematics and Systems Science, Chinese Academy of Sciences, 100190, Beijing, China; Mathematical Sciences, University of Chinese Academy of Sciences, Beijing 100049, China, Email: {dyh@lsec.cc.ac.cn} }~~~
        Yakui Huang  \thanks{Institute of Mathematics, Hebei University of Technology, Tianjin 300401, China, Email: {huangyakui2006@gmail.com}}~~~
        Xin-Wei Liu \thanks{Institute of Mathematics, Hebei University of Technology, Tianjin 300401, China, Email: {optim2008@163.com}}
}

\date{}
%

\maketitle

\begin{abstract}
\noindent
We propose a family of spectral gradient methods, whose stepsize is determined by a convex combination of the long Barzilai-Borwein (BB) stepsize and the short BB stepsize. Each member of the family is shown to share certain quasi-Newton property in the sense of least squares. The family also includes some other gradient methods as its special cases. We prove that the family of methods is $R$-superlinearly convergent for two-dimensional strictly convex quadratics. Moreover, the family is $R$-linearly convergent in the any-dimensional case. Numerical results of the family with different settings are presented, which demonstrate that the proposed family is promising.\\

\end{abstract}

\section{Introduction}
\label{intro}
Consider the unconstrained optimization problem
\begin{equation}\label{eqpro}
  \min_{x\in\mathbb{R}^n}~~f(x),
\end{equation}
where $f(x): \mathbb{R}^n\rightarrow\mathbb{R}$ is a continuously differentiable function. The gradient method solves problem \eqref{eqpro} by updating iterates as
\begin{equation}\label{eqitr}
  x_{k+1}=x_k-\alpha_kg_k,
\end{equation}
where $g_k=\nabla f(x_k)$ and $\alpha_k>0$ is the stepsize. Different gradient methods use different formulae for stepsizes.

The simplest gradient method is the steepest descent (SD) method due to Cauchy \cite{cauchy1847methode}, which computes the stepsize by exact line search,
\begin{equation*}\label{eqsd}
  \alpha_k^{SD}=\arg\min_{\alpha\in \mathbb{R}}~f(x_k-\alpha g_k).
\end{equation*}
As is well known, every two consecutive gradients generated by the SD method are perpendicular to each other. Moreover, if $f(x)$ is a strictly convex quadratic function, i.e.,
\begin{equation}\label{qdpro}
  f(x)=\frac{1}{2}x^TAx-b^Tx,
\end{equation}
where $A\in \mathbb{R}^{n\times n}$ is symmetric positive definite and $b\in \mathbb{R}^n$, it can be shown that the gradients will asymptotically reduce to a two-dimensional subspace spanned by the two eigenvectors corresponding to the largest and smallest eigenvalues of the matrix $A$ and hence zigzag occurs, see \cite{akaike1959successive,nocedal2002behavior} for more details. This property seriously deteriorates the performance of the SD method, especially when the condition number of $A$ is large.

An important approach that changes our perspectives on the effectiveness of gradient methods is proposed by Barzilai and Borwein \cite{Barzilai1988two}. They viewed the updating rule \eqref{eqitr} as
\begin{equation}\label{eqitr2}
  x_{k+1}=x_k-D_kg_k,
\end{equation}
where $D_k=\alpha_kI$. Similar to the quasi-Newton method \cite{dennis1977quasi}, $D_k^{-1}$ is required to satisfy the secant equation
\begin{equation}\label{qsb}
  B_ks_{k-1}=y_{k-1}
\end{equation}
to approximate the Hessian as possible as it can. Here, $s_{k-1}=x_k-x_{k-1}$ and $y_{k-1}=g_k-g_{k-1}$. However, since $D_k$ is diagonal with identical diagonal elements, it is usually impossible to find an $\alpha_k$ such that $D_k^{-1}$ fulfills \eqref{qsb} if the dimension $n>1$. Thus, Barzilai and Borwein required $D_k^{-1}$ to meet the secant equation in the sense of least squares,
\begin{equation}\label{qs1}
  D_k=\arg\min_{D=\alpha I}~\|D^{-1}s_{k-1}-y_{k-1}\|,
\end{equation}
which yields
\begin{equation}\label{bb1}
  \alpha_k^{BB1}=\frac{s_{k-1}^Ts_{k-1}}{s_{k-1}^Ty_{k-1}}. 
\end{equation}
Here and below, $\|\cdot\|$ means the Euclidean norm.
On the other hand, one can also calculate the stepsize by requiring $D_k$ to satisfy
\begin{equation}\label{qsh}
  H_ky_{k-1}=s_{k-1}.
\end{equation}
That is,
\begin{equation}\label{qs2}
  D_k=\arg\min_{D=\alpha I}~\|s_{k-1}-Dy_{k-1}\|,
\end{equation}
which gives
\begin{equation}\label{bb2}
  \alpha_k^{BB2}=\frac{s_{k-1}^Ty_{k-1}}{y_{k-1}^Ty_{k-1}}.
\end{equation}
 Apparently, when $s_{k-1}^Ty_{k-1}>0$, there holds $\alpha_{k}^{BB1}\geq\alpha_{k}^{BB2}$. In other words, $\alpha_{k}^{BB1}$ is a {\it long} stepsize while $\alpha_{k}^{BB2}$ is a {\it short} one, which implies that $\alpha_{k}^{BB1}$ is more aggressive than $\alpha_{k}^{BB2}$ in decreasing the objective value. Extensive numerical experiments show that the long stepsize is superior to the short one in many cases, see \cite{birgin2014spectral,fletcher2005barzilai,raydan1997barzilai} for example. In what follows we will refer to the gradient method with the long stepsize $\alpha_{k}^{BB1}$ as the BB method without specification.


Barzilai and Borwein \cite{Barzilai1988two} proved their method with the short BB stepsize $\alpha_{k}^{BB2}$ is $R$-superlinearly convergent for two-dimensional strictly convex quadratics.
An improved $R$-superlinear convergence result for the BB method was given by Dai \cite{dai2013new}. Global and $R$-linear convergence of the BB method for general $n$-dimensional strictly convex quadratics were established by Raydan \cite{raydan1993Barzilai} and Dai and Liao \cite{dai2002r}, respectively. The BB method has also been extended to solve general nonlinear optimization problems. By incorporating the nonmontone line search proposed by Grippo et al. \cite{grippo1986nonmonotone}, Raydan \cite{raydan1997barzilai} developed the global BB method for general unconstrained problems. Later, Birgin et al. \cite{birgin2000nonmonotone} proposed the so-called spectral projected gradient method which extends Raydan's method to smooth convex constrained problems. Dai and Fletcher \cite{dai2005projected} designed projected BB methods for large-scale box-constrained quadratic programming. Recently, by resorting to the smoothing techniques, Huang and Liu \cite{huang2016smoothing} generalized the projected BB method with modifications to solve non-Lipschitz optimization problems.

{{
The relationship between the stepsizes in BB-like methods and the spectrum of the Hessian of the objective function has been explored in several studies. Frassoldati et al. \cite{frassoldati2008new} tried to exploit the long BB stepsize close to the reciprocal of the smallest eigenvalue of the Hessian, yielding the ABBmin1 and ABBmin2 methods. De Asmundis et al. \cite{de2013spectral} developed the so-called SDA method which employs a short stepsize approximates the reciprocal of the largest eigenvalue of the Hessian. Following the line of \cite{de2013spectral}, Gonzaga and Schneider \cite{gonzaga2016steepest} suggested a monotone method for quadratics where the stepsizes are obtained in a way similar to the SD method. De Asmundis et al. \cite{de2014efficient} proposed the SDC method which exploits the spectral property of Yuan's stepsize \cite{dai2005analysis}.
Kalousek \cite{kalousek2017steepest} considered the SD method with random stepsizes lying between the reciprocal of the largest eigenvalue and the smallest eigenvalue of the Hessian and analysed the optimal distribution of random stepsizes that guarantees the maximum asymptotic convergence rate.
}}

%
%


Applications of the BB method and its variants have largely been developed for problems arising in various different areas including image restoration \cite{wang2007projected}, signal processing \cite{liu2011coordinated}, eigenvalue problems \cite{jiang2013feasible}, nonnegative matrix factorization \cite{huang2015quadratic}, sparse reconstruction \cite{wright2009sparse}, machine learning \cite{tan2016barzilai}, etc. We refer the reader to
\cite{birgin2014spectral,dai2003alternate,dai2016barzilai,di2018steplength,fletcher2005barzilai,yuan2008step} and references therein for more spectral gradient methods and extensions.

The success of the BB method and its variants motivates us to consider spectral gradient methods. Our goal is to present a family of spectral gradient methods for optimization. Notice that the Broyden class of quasi-Newton methods \cite{broyden1965class} approximate the inverse of the Hessian  by
\begin{equation}\label{Broydenh}
  H_{k}^\tau=\tau H_k^{BFGS} + (1-\tau)H_k^{DFP},
\end{equation}
where $\tau\in[0,1]$ is a scalar parameter and $H_k^{BFGS}$ and $H_k^{DFP}$ are the BFGS and DFP matrices, respectively, that satisfy the secant equation \eqref{qsh}, which further implies that
\begin{equation*}
\tau H_k^{BFGS}y_{k-1} + (1-\tau)H_k^{DFP}y_{k-1}=s_{k-1},
\end{equation*}
i.e.,
\begin{equation}\label{Broydenm}
\tau(H_k^{BFGS}y_{k-1}-s_{k-1}) + (1-\tau)(H_k^{DFP}y_{k-1}-s_{k-1})=0.
\end{equation}
Since the inverse of $H_k^{BFGS}$, say $B_k^{BFGS}$, satisfies \eqref{qsb}, we can modify \eqref{Broydenm} as
\begin{equation}\label{Broydenm2}
\tau(B_k^{BFGS}s_{k-1}-y_{k-1}) + (1-\tau)(s_{k-1}-H_k^{DFP}y_{k-1})=0.
\end{equation}
Motivated by the above observation, we employ the idea of the BB method to approximate the Hessian  and its inverse by diagonal matrices. Particularly, we require the matrix $D=\alpha I$ to be the solution of
 \begin{equation}\label{qsstep1}
  \min_{D=\alpha I}~\|\tau(D^{-1}s_{k-1}-y_{k-1})+(1-\tau)(s_{k-1}-D y_{k-1})\|.
\end{equation}
In the next section, we will show that the stepsize given by the convex combination of the long BB stepsize $\alpha_k^{BB1}$ and the short BB stepsize $\alpha_k^{BB2}$, i.e.,
\begin{equation}\label{bbc}
  \alpha_k=\gamma_k \alpha_{k}^{BB1}+(1-\gamma_k)\alpha_{k}^{BB2},
\end{equation}
where $\gamma_k\in[0,1]$, is a solution to \eqref{qsstep1}. Clearly, this is a one-parametric family of stepsizes, which include the two BB stepsizes as special instances. Moreover, any stepsize lies in the interval $[\alpha_{k}^{BB2},\alpha_{k}^{BB1}]$ is a special case of the family. For example, the positive stepsize given by the geometrical mean of $\alpha_{k}^{BB1}$ and $\alpha_{k}^{BB2}$ \cite{dai2015positive},
\begin{equation}\label{ap}
  \alpha_k^{P}=\sqrt{\alpha_{k}^{BB1}\alpha_{k}^{BB2}}=\frac{\|s_{k-1}\|}{\|y_{k-1}\|}.
\end{equation}
We further prove that the family of spectral gradient methods \eqref{bbc} is $R$-superlinearly convergent for two-dimensional strictly convex quadratics. For the $n$-dimensional case, the family is proved to be $R$-linearly convergent. Numerical results of the family \eqref{bbc} with different settings of $\gamma_k$
{{are presented and compared with other gradient methods, including the BB  method \cite{Barzilai1988two}, the alternate BB method (ALBB) \cite{dai2005projected}, the adaptive BB method (ABB) \cite{zhou2006gradient}, the cyclic BB method with stepsize $\alpha_k^{BB1}$ (CBB1) \cite{dai2006cyclic}, the cyclic BB method with stepsize $\alpha_k^{BB2}$ (CBB2), the cyclic method with stepsize $\alpha_k^{P}$ (CP), the Dai-Yuan method (DY) \cite{dai2005analysis}, the ABBmin1 and ABBmin2 methods \cite{frassoldati2008new}, and the SDC method \cite{de2014efficient}. The comparisons demonstrate that the proposed family is promising.
}}

The paper is organized as follows. In Section 2, we show that each stepsize in the family \eqref{bbc} solves some least squares problem \eqref{qsstep1} and hence possesses certain quasi-Newton property. In Section 3, we establish $R$-superlinear convergence of the family \eqref{bbc} for two-dimensional strictly convex quadratics and $R$-linear convergence for the $n$-dimensional case, respectively. In Section 4, we discuss different selection rules for the parameter $\gamma_k$. In Section 5, we conduct some numerical comparisons of our approach and {{other gradient methods}}. Finally, some conclusions are drawn in Section 6.

\section{Quasi-Newton property of the family \eqref{bbc}}

In this section, we show that each stepsize in the family \eqref{bbc} enjoys certain quasi-Newton property.


For the sake of simplicity, we discard the subscript of $s_{k-1}$ and $y_{k-1}$ in the following part of this section, i.e., $s=s_{k-1}$, $y=y_{k-1}$. Let
\begin{equation*}
  \phi_\tau(\alpha):=\|\tau(\frac{1}{\alpha}s-y)+(1-\tau)(s-\alpha y)\|^2.
\end{equation*}
Then, the derivative of $\phi_\tau(\alpha)$ with respect to $\alpha$ is
\begin{align*}
  \phi_\tau'(\alpha)
  &=2[\tau+(1-\tau)\alpha]\{(-\frac{\tau}{\alpha^3})s^Ts
-[(1-\tau)\frac{1}{\alpha}
  +(-\frac{\tau}{\alpha^2})]s^Ty+(1-\tau)y^Ty\}.
\end{align*}

\begin{proposition}\label{pro1}
  {{If $s^Ty>0$ and $\tau\in[0,1]$}}, the equation $\phi'_\tau(\alpha)=0$ has a {{unique}} root in $[\alpha_k^{BB2},\alpha_k^{BB1}]$.
\end{proposition}
\begin{proof}
We only need to consider the case $\tau\in(0,1)$. Notice that
\begin{align}\label{eqdefh}
\psi(\tau,\alpha):& =\frac{1}{2}\frac{\alpha^3}{\tau+(1-\tau)\alpha}\phi_\tau'(\alpha) \nonumber\\
  &=-\tau s^Ts-[(1-\tau)\alpha^2-\tau\alpha]s^Ty+(1-\tau)\alpha^3y^Ty\nonumber\\
   &=(1-\tau)(\alpha^3y^Ty-\alpha^2s^Ty)+\tau(\alpha s^Ty-s^Ts)\nonumber\\
   &=(1-\tau)y^Ty(\alpha^3-\alpha^2\alpha_k^{BB2})+\tau s^Ty(\alpha-\alpha_k^{BB1}).
\end{align}
{{If $s^Ty>0$, we have $y\ne 0$ and $y^Ty>0$. This implies that}} $\psi(\tau,\alpha_k^{BB2})<0$ and $\psi(\tau,\alpha_k^{BB1})>0$. Thus, $\psi(\tau,\alpha)=0$ has a root in $(\alpha_k^{BB2},\alpha_k^{BB1})$. Since $\alpha>0$, we {{know that the equation $\phi'_\tau(\alpha)=0$ has a root in $[\alpha_k^{BB2},\alpha_k^{BB1}]$.}}

{{Now we show the uniqueness of such a root by contradiction. Suppose that $\alpha_1<\alpha_2$ and $\alpha_1,\alpha_2\in[\alpha_k^{BB2},\alpha_k^{BB1}]$ such that $\phi'_\tau(\alpha_1)=0$ and $\phi'_\tau(\alpha_2)=0$. It follows from \eqref{eqdefh} that
\begin{align*}
 & (1-\tau)y^Ty(\alpha_1^3-\alpha_1^2\alpha_k^{BB2})+\tau s^Ty(\alpha_1-\alpha_k^{BB1}) \\
  =&(1-\tau)y^Ty(\alpha_2^3-\alpha_2^2\alpha_k^{BB2})+\tau s^Ty(\alpha_2-\alpha_k^{BB1}).
\end{align*}
By rearranging terms, we obtain
\begin{align*}
 & (1-\tau)y^Ty[(\alpha_1-\alpha_2)(\alpha_1^2+\alpha_1\alpha_2+\alpha_2^2)
 -(\alpha_1-\alpha_2)(\alpha_1+\alpha_2)\alpha_k^{BB2}]\\
 & =\tau s^Ty(\alpha_2-\alpha_1).
\end{align*}
Since $\alpha_1\ne \alpha_2$, it follows that
\begin{equation*}\label{eqpro1}
  (1-\tau)y^Ty[(\alpha_1^2+\alpha_1\alpha_2+\alpha_2^2)
 -(\alpha_1+\alpha_2)\alpha_k^{BB2}]=-\tau s^Ty,
\end{equation*}
which gives $(\alpha_1^2+\alpha_1\alpha_2+\alpha_2^2)
 -(\alpha_1+\alpha_2)\alpha_k^{BB2}<0$. This is not possible since $\alpha_k^{BB2}\leq\alpha_1<\alpha_2$. This completes the proof.
}}
\qed
\end{proof}

\begin{proposition}\label{pro2}
 {{If $s^Ty>0$ and $\tau\in[0,1]$}},
 the root of $\phi'_\tau(\alpha)=0$ in $[\alpha_k^{BB2},\alpha_k^{BB1}]$ is monotone with respect to $\tau$.
\end{proposition}
\begin{proof}
It suffices to show the statement holds for $\tau\in(0,1)$. By the proof of Proposition \ref{pro1}, $\alpha$ is an implicit function of $\tau$ determined by the equation $\psi(\tau,\alpha)=0$. The derivative of $\psi(\tau,\alpha)$ with respect to $\tau$ is
\begin{align*}
  \frac{\partial \psi(\tau,\alpha)}{\partial \tau}=-y^Ty(\alpha^3-\alpha^2\alpha_k^{BB2})
&+
  (1-\tau)y^Ty(3\alpha^2\cdot\alpha'-2\alpha_k^{BB2}\alpha\cdot\alpha')\\
  &+
  s^Ty(\alpha-\alpha_k^{BB1})+
  \tau s^Ty\alpha'.
\end{align*}
Let $\frac{\partial \psi(\tau,\alpha)}{\partial \tau}=0$. By simple calculations, we obtain
\begin{equation*}
  \alpha'=\frac{y^Ty(\alpha^3-\alpha^2\alpha_k^{BB2})-s^Ty(\alpha-\alpha_k^{BB1})}
  {(1-\tau)y^Ty(3\alpha^2-2\alpha_k^{BB2}\alpha)+
  \tau s^Ty}.
\end{equation*}
For $\alpha\in(\alpha_k^{BB2},\alpha_k^{BB1})$, $\alpha'>0$. This completes the proof.
\qed
\end{proof}

\begin{theorem}\label{th1}
For each $\gamma_k\in[0,1]$, the stepsize $\alpha_k$ defined by \eqref{bbc} is a solution of \eqref{qsstep1}.
\end{theorem}
\begin{proof}
We only need to show that, for $\gamma_k\in(0,1)$, $\phi'_{\tau}(\alpha_k)$ vanishes at some $\tilde{\tau}\in(0,1)$. From \eqref{eqdefh} and \eqref{bbc}, we have
\begin{align*}
\psi(\tau,\alpha_k)
   &=(1-\tau)\alpha_k^2y^Ty(\alpha_k-\alpha_k^{BB2})+\tau s^Ty(\alpha_k-\alpha_k^{BB1})\\
   &=(1-\tau)\gamma_k\alpha_k^2y^Ty(\alpha_k^{BB1}-\alpha_k^{BB2})+\tau(\gamma_k-1) s^Ty(\alpha_k^{BB1}-\alpha_k^{BB2})\\
   &=y^Ty[(1-\tau)\gamma_k\alpha_k^2+\tau(\gamma_k-1)\alpha_k^{BB2}](\alpha_k^{BB1}-\alpha_k^{BB2}).
\end{align*}
Clearly,
\begin{equation*}
  \tilde{\tau}=\frac{\gamma_k\alpha_k^2}{\gamma_k\alpha_k^2+(1-\gamma_k)\alpha_k^{BB2}}\in(0,1)
\end{equation*}
is a root of $\psi(\tau,\alpha_k)=0$. This completes the proof.
\qed
\end{proof}

\section{Convergence analysis}
{{In this section, we analyze the convergence properties of the family \eqref{bbc} for the quadratic function \eqref{qdpro}.}}
Since the gradient method \eqref{eqitr} is invariant under translations and rotations when applying to problem \eqref{qdpro}, we assume that the matrix $A$ is diagonal, i.e.,
\begin{equation}\label{formA}
  A=\textrm{diag}\{\lambda_1,\lambda_2,\cdots,\lambda_n\},
\end{equation}
where $0<\lambda_1\leq\lambda_2\leq\cdots\leq\lambda_n$.

\subsection{Two-dimensional case}
In this subsection, based on the techniques in \cite{dai2015positive}, we establish the $R$-superlinear convergence of the family \eqref{bbc} for two-dimensional quadratic functions.

Without loss of generality, we assume that
\begin{equation*}
  A=\left(
      \begin{array}{cc}
        1 & ~0 \\
        0 & ~\lambda \\
      \end{array}
    \right),~~b=0,
\end{equation*}
where $\lambda\geq1$. Furthermore, assume that $x_1$ and $x_2$ are such that
\begin{equation}\label{initialg}
  g_1^{(i)}\neq0,~~g_2^{(i)}\neq0,~~i=1,2.
\end{equation}
Let
\begin{equation*}
  q_k=\frac{(g_k^{(1)})^2}{(g_k^{(2)})^2}.
\end{equation*}
Then it follows that
\begin{equation*}
  \alpha_{k}^{BB1}
  =\frac{g_{k-1}^Tg_{k-1}}{g_{k-1}^TAg_{k-1}}
  =\frac{1+q_{k-1}}{\lambda+q_{k-1}},
\end{equation*}
\begin{equation*}
  \alpha_{k}^{BB2}
  =\frac{g_{k-1}^TAg_{k-1}}{g_{k-1}^TA^2g_{k-1}}
  =\frac{\lambda+q_{k-1}}{\lambda^2+q_{k-1}}.
\end{equation*}
Thus, the stepsize \eqref{bbc} can be written as
\begin{align}\label{bbc2}
  \alpha_k&=\gamma_k\frac{1+q_{k-1}}{\lambda+q_{k-1}}+
  (1-\gamma_k)\frac{\lambda+q_{k-1}}{\lambda^2+q_{k-1}}\nonumber\\
  &=\frac{\gamma_k(1+q_{k-1})(\lambda^2+q_{k-1})+
  (1-\gamma_k)(\lambda+q_{k-1})^2}{(\lambda+q_{k-1})(\lambda^2+q_{k-1})}.
\end{align}
By the update rule \eqref{eqitr} and $g_k=Ax_k$, we have
\begin{equation*}
  g_{k+1}=(I-\alpha_kA)g_k.
\end{equation*}
Thus,
\begin{align}\label{eqg1}
  (g_{k+1}^{(1)})^2&=(1-\alpha_k)^2(g_k^{(1)})^2\nonumber\\
  &=\frac{(\lambda-1)^2\left[\gamma_k(\lambda^2+q_{k-1})+
  (1-\gamma_k)\lambda(\lambda+q_{k-1})\right]^2}{(\lambda+q_{k-1})^2(\lambda^2+q_{k-1})^2}
  (g_k^{(1)})^2,
\end{align}
\begin{align}\label{eqg2}
  (g_{k+1}^{(2)})^2&=(1-\lambda\alpha_k)^2(g_k^{(2)})^2\nonumber\\
  &=\frac{q_{k-1}^2(1-\lambda)^2  \left[\gamma_k(\lambda^2+q_{k-1})+
  (1-\gamma_k)(\lambda+q_{k-1})\right]^2}{(\lambda+q_{k-1})^2(\lambda^2+q_{k-1})^2}
  (g_k^{(2)})^2.
\end{align}
From \eqref{eqg1}, \eqref{eqg2} and the definition of $q_k$, we get
\begin{equation}\label{qk1}
  q_{k+1}=\frac{(g_{k+1}^{(1)})^2}{(g_{k+1}^{(2)})^2}
  =\left(\frac{\gamma_k(\lambda^2+q_{k-1})+
  (1-\gamma_k)\lambda(\lambda+q_{k-1})}{\gamma_k(\lambda^2+q_{k-1})+
  (1-\gamma_k)(\lambda+q_{k-1})}\right)^2\frac{q_k}{q_{k-1}^2}.
\end{equation}

Let
\begin{equation*}
  h_k(w)=\frac{\gamma_k(\lambda^2+w)+
  (1-\gamma_k)\lambda(\lambda+w)}{\gamma_k(\lambda^2+w)+
  (1-\gamma_k)(\lambda+w)}.
\end{equation*}
Then we have
\begin{equation*}
  h_k(0)=\frac{\gamma_k\lambda^2+
  (1-\gamma_k)\lambda^2}{\gamma_k\lambda^2+
  (1-\gamma_k)\lambda}=\frac{\lambda}{\gamma_k\lambda+1-\gamma_k},
\end{equation*}
\begin{equation*}
  \lim_{w\rightarrow\infty}h_k(w)=\frac{\gamma_k+
  (1-\gamma_k)\lambda}{\gamma_k+
  1-\gamma_k}=\gamma_k+(1-\gamma_k)\lambda.
\end{equation*}
Since
\begin{equation}\label{dhw}
  h'_k(w)=
  \frac{\gamma_k(1-\gamma_k)\lambda(\lambda-1)^2}{(\gamma_k(\lambda^2+w)+
  (1-\gamma_k)(\lambda+w))^2},
\end{equation}
we obtain $h'(w)>0$ for $\gamma_k\in(0,1)$. Thus,
\begin{equation}\label{hwbd}
 h_k(w)\in\left(\frac{\lambda}{\gamma_k\lambda+1-\gamma_k},\gamma_k+(1-\gamma_k)\lambda\right).
\end{equation}
Denoting $M_k=\log q_k$. By \eqref{qk1}, we have
\begin{equation}\label{eqmrec}
  M_{k+1}=M_k-2M_{k-1}+2\log h_k(q_{k-1}).
\end{equation}
Let $\theta$ such that $\theta^2-\theta+2=0$.
Then, $\theta=\frac{1\pm\sqrt{7}i}{2}$. Denote by
\begin{equation}\label{defkesi}
  \xi_k=M_k+(\theta-1)M_{k-1}.
\end{equation}
We have the following result.
\begin{lemma}\label{lm1}
  If $\gamma_k\in(0,1)$ and
\begin{equation}\label{eqkesi2}
  |\xi_2|>8\log \lambda,
\end{equation}
there exists $c_1>0$ such that
\begin{equation}\label{eqbdkesi}
  |\xi_k|\geq (\sqrt{2}-1)2^{\frac{k}{2}}c_1,~~k\geq2.
\end{equation}
\end{lemma}
\begin{proof}
It follows from \eqref{eqmrec}, the definition of $\theta$, and \eqref{defkesi} that
\begin{equation}\label{kesip1}
  \xi_{k+1}=\theta M_k-2M_{k-1}+2\log h_k(q_{k-1})=\theta\xi_k+2\log h_k(q_{k-1}).
\end{equation}
By \eqref{hwbd}, we know that
\begin{equation*}
  0<\log h_k(q_{k-1})<\log (\gamma_k+(1-\gamma_k)\lambda)<\log \lambda.
\end{equation*}
Since $|\theta|=\sqrt{2}$, we get by \eqref{kesip1} that
\begin{equation*}
  |\xi_{k+1}|\geq\sqrt{2}\,|\xi_k|-c_1,
\end{equation*}
where $c_1=2\log \lambda$. From \eqref{eqkesi2}, we have
\begin{align*}
  |\xi_{k+1}|&\geq2^{\frac{k-1}{2}}|\xi_2|-\frac{2^{\frac{k}{2}}-1}{\sqrt{2}-1}c_1
  \geq(2^{\frac{k+3}{2}}-\frac{2^{\frac{k}{2}}-1}{\sqrt{2}-1})c_1\\
  &=[(\sqrt{2}-1)(2^{\frac{k}{2}}+1)+2]c_1
  >(\sqrt{2}-1)2^{\frac{k}{2}}c_1.
\end{align*}
This completes the proof.
\qed
\end{proof}

Since $|\theta-1|=\sqrt{2}$, we obtain by \eqref{defkesi} that
\begin{align*}
  |\xi_k|&\leq |M_k|+|\theta-1||M_{k-1}|=|M_k|+\sqrt{2}|M_{k-1}|\\
&\leq(\sqrt{2}+1)\max\{|M_k|,|M_{k-1}|\}£¬
\end{align*}
which, together with \eqref{eqbdkesi}, gives
\begin{equation}\label{bdm}
  \max\{|M_k|,|M_{k-1}|\}\geq\frac{1}{\sqrt{2}+1}(\sqrt{2}-1)2^{\frac{k}{2}}c_1
  =(\sqrt{2}-1)^22^{\frac{k}{2}}c_1.
\end{equation}

\begin{lemma}\label{lm2}
Under the conditions of Lemma \ref{lm1}, for $k\geq2$, the following inequalities hold:
\begin{equation}\label{upbdm}
  \max_{-1\leq i\leq3}M_{k+i}\geq(\sqrt{2}-1)^22^{\frac{k}{2}}c_1-2c_1,
\end{equation}
\begin{equation}\label{lwbdm}
  \min_{-1\leq i\leq3}M_{k+i}\leq-(\sqrt{2}-1)^22^{\frac{k}{2}}c_1+2c_1.
\end{equation}
\end{lemma}
\begin{proof}
{{The inequality \eqref{upbdm} holds true if
\begin{equation*}
  M_{k-1}\geq(\sqrt{2}-1)^22^{\frac{k}{2}}c_1
\end{equation*}
or
\begin{equation*}
  M_{k}\geq(\sqrt{2}-1)^22^{\frac{k}{2}}c_1.
\end{equation*}
Suppose that the above two inequalities are false. By \eqref{bdm}, we know that either
\begin{equation*}
  M_{k-1}\leq-(\sqrt{2}-1)^22^{\frac{k}{2}}c_1
\end{equation*}
or
\begin{equation*}
  M_{k}\leq-(\sqrt{2}-1)^22^{\frac{k}{2}}c_1.
\end{equation*}
From \eqref{eqmrec}, we have
\begin{align}\label{eqm2}
  M_{k+2}&=M_{k+1}-2M_k+2\log h(q_k)\nonumber\\
  &=-M_k-2M_{k-1}+2\log h(q_k)+2\log h_k(q_{k-1}).
\end{align}
(i) $M_{k-1}\leq-(\sqrt{2}-1)^22^{\frac{k}{2}}c_1$.
If $M_{k}<0$, it follows from \eqref{eqm2} that
\begin{equation*}
  M_{k+2}\geq-2M_{k-1}+2\log h_k(q_{k-1})\geq(\sqrt{2}-1)^22^{\frac{k}{2}}c_1-2c_1.
\end{equation*}
Otherwise, if $M_{k}\geq0$, by \eqref{eqmrec}, we get
\begin{equation*}
  M_{k+1}\geq-2M_{k-1}+2\log h_k(q_{k-1})\geq(\sqrt{2}-1)^22^{\frac{k}{2}}c_1-2c_1.
\end{equation*}
(ii) $M_{k}\leq-(\sqrt{2}-1)^22^{\frac{k}{2}}c_1$. Similar to (i), we can show that
\begin{equation*}
  M_{k+3}\geq(\sqrt{2}-1)^22^{\frac{k}{2}}c_1-2c_1
\end{equation*}
or
\begin{equation*}
  M_{k+2}\geq(\sqrt{2}-1)^22^{\frac{k}{2}}c_1-2c_1.
\end{equation*}
Thus, the inequality \eqref{upbdm} is valid.
The inequality \eqref{lwbdm} can be established in a similar way.}}
\qed
\end{proof}

\begin{theorem}\label{th2}
  If $\gamma_k\in(0,1)$, \eqref{initialg} and \eqref{eqkesi2} hold, the sequence $\{\|g_k\|\}$ converges to zero $R$-superlinearly.
\end{theorem}
\begin{proof}
From \eqref{eqg2}, we have
\begin{align}\label{gk12up}
  |g_{k+1}^{(2)}|
  &=\frac{q_{k-1}(\lambda-1) \left[\gamma_k(\lambda^2+q_{k-1})+
  (1-\gamma_k)(\lambda+q_{k-1})\right]}{(\lambda+q_{k-1})(\lambda^2+q_{k-1})}
  |g_k^{(2)}|\nonumber\\
  &\leq\frac{q_{k-1}(\lambda-1)(\lambda^2+q_{k-1})}{(\lambda+q_{k-1})(\lambda^2+q_{k-1})}
  |g_k^{(2)}|\nonumber\\
  &\leq(\lambda-1)q_{k-1}|g_k^{(2)}|.
\end{align}
Since $\alpha_k\in(\lambda^{-1},1)$, we have
\begin{equation}\label{gkiup}
|g_{k+1}^{(i)}|\leq(\lambda-1)|g_k^{(i)}|,~~i=1,2,
\end{equation}
{{which gives
\begin{equation}\label{gkiup2}
|g_{k+5}^{(i)}|\leq(\lambda-1)^{(5-j)}|g_{k+j}^{(i)}|,~~i=1,2,~j=1,\ldots,5.
\end{equation}
It follows from \eqref{gk12up}, \eqref{gkiup} and \eqref{gkiup2} that
\begin{align}\label{gkiup3}
|g_{k+5}^{(2)}|&\leq(\lambda-1)^{(5-j+1)}q_{k+j-2}|g_{k+j-1}^{(2)}|\nonumber\\
&\leq(\lambda-1)^{5}q_{k+j-2}|g_{k}^{(2)}|
,~~i=1,2,~j=1,\ldots,5,
\end{align}
which indicates that
\begin{equation}\label{gk5up}
|g_{k+5}^{(2)}|\leq(\lambda-1)^5\left(\min_{-1\leq i\leq3}q_{k+i}\right)|g_k^{(2)}|.
\end{equation}
}}
As $M_k=\log q_k$, we know by Lemma \ref{lm2} and \eqref{gk5up} that
\begin{equation*}
|g_{k+5}^{(2)}|\leq(\lambda-1)^5\exp\left(-(\sqrt{2}-1)^22^{\frac{k}{2}}c_1+2c_1\right)|g_k^{(2)}|.
\end{equation*}
Similarly to \eqref{gk12up}, we have
\begin{align*}
  |g_{k+1}^{(1)}|
  &=\frac{(\lambda-1) \left[\gamma_k(\lambda^2+q_{k-1})+
  (1-\gamma_k)\lambda(\lambda+q_{k-1})\right]}{(\lambda+q_{k-1})(\lambda^2+q_{k-1})}
  |g_k^{(1)}|\\
  &\leq\frac{\lambda(\lambda-1)(\lambda+q_{k-1})}{(\lambda+q_{k-1})(\lambda^2+q_{k-1})}
  |g_k^{(1)}|\\
  &<\lambda(\lambda-1)\frac{1}{q_{k-1}}|g_k^{(1)}|,
\end{align*}
which together with {{\eqref{gkiup2}}} yields that
\begin{align}\label{gk11up}
|g_{k+5}^{(1)}|&\leq\lambda(\lambda-1)^5\frac{1}{\max\limits_{-1\leq i\leq3}{{q_{k+i}}}}|g_k^{(1)}|\nonumber\\
&\leq\lambda(\lambda-1)^5\exp\left(-(\sqrt{2}-1)^22^{\frac{k}{2}}c_1+2c_1\right)|g_k^{(1)}|.
\end{align}
By \eqref{gk12up} and \eqref{gk11up}, for any $k$, we have
\begin{equation*}
\|g_{k+5}\|\leq\lambda(\lambda-1)^5\exp\left(-(\sqrt{2}-1)^22^{\frac{k}{2}}c_1+2c_1\right)\|g_k\|.
\end{equation*}
That is, $\{\|g_k\|\}$ converges to zero $R$-superlinearly.
\qed
\end{proof}

Theorem \ref{th2}, together with the analysis for the BB method in \cite{Barzilai1988two,dai2013new}, shows that, for any $\gamma_k\in[0,1]$, the family \eqref{bbc} is $R$-superlinearly convergent.



\subsection{$n$-dimensional case}
In this subsection, we show $R$-linear convergence of the family \eqref{bbc} for $n$-dimensional quadratics.

Dai \cite{dai2003alternate} has proved that if $A$ has the form \eqref{formA} with $1=\lambda_1\leq\lambda_2\leq\cdots\leq\lambda_n$ and the stepsizes of gradient method \eqref{eqitr} have the following Property (A), then either $g_k=0$ for some finite $k$ or the sequence $\{\|g_k\|\}$ converges to zero $R$-linearly.

\noindent
\textbf{Property (A)} \cite{dai2003alternate}. Suppose that there exist an integer $m$ and positive constants $M_1\geq\lambda_1$ and $M_2$ such that
\begin{itemize}
  \item[(i)] $\lambda_1\leq\alpha_k^{-1}\leq M_1$;
  \item[(ii)] for any integer $l\in[1,n-1]$ and $\epsilon>0$, if $G(k-j,l)\leq\epsilon$ and $(g_{k-j}^{(l+1)})^2\geq M_2\epsilon$ hold for $j\in[0,\min\{k,m\}-1]$, then $\alpha_k^{-1}\geq\frac{2}{3}\lambda_{l+1}$.
\end{itemize}
Here,
\begin{equation*}
  G(k,l)=\sum_{i=1}^l(g_k^{(i)})^2.
\end{equation*}

Now we show $R$-linear convergence of the proposed family by proving that the stepsize \eqref{bbc} satisfies Property (A).
\begin{theorem}\label{th3}
Suppose that the sequence $\{\|g_k\|\}$ is generated by the family \eqref{bbc} applied to $n$-dimensional quadratics with the matrix $A$ has the form \eqref{formA} and $1=\lambda_1\leq\lambda_2\leq\cdots\leq\lambda_n$. Then either $g_k=0$ for some finite $k$ or the sequence $\{\|g_k\|\}$ converges to zero $R$-linearly.
\end{theorem}
\begin{proof}
Let $M_1=\lambda_n$ and $M_2=2$. We have by \eqref{bbc} and the fact $\gamma_k\in[0,1]$ that
\begin{equation*}
  \alpha_{k}^{BB2}\leq\alpha_k\leq\alpha_{k}^{BB1}.
\end{equation*}
Thus, (i) of Property (A) holds. If $G(k-j,l)\leq\epsilon$ and $(g_{k-j}^{(l+1)})^2\geq M_2\epsilon$ hold for $j\in[0,\min\{k,m\}-1]$, we have
\begin{align*}\label{alfinvup}
  \alpha_{k}^{-1}&\geq\frac{1}{\alpha_{k}^{BB1}}=\frac{1}{\alpha_{k-1}^{SD}}
  =\frac{\sum_{i=1}^n\lambda_i(g^{(i)}_{k-1})^2}{\sum_{i=1}^n(g^{(i)}_{k-1})^2}\nonumber\\
&\geq\frac{\lambda_{l+1}\sum_{i=l+1}^n(g^{(i)}_{k-1})^2}{\epsilon\|g_{k-1}\|^2+\sum_{i=l+1}^n(g^{(i)}_{k-1})^2}\nonumber\\
&\geq\frac{M_2}{M_2+1}\lambda_{l+1}
=\frac{2}{3}\lambda_{l+1}.
\end{align*}
That is, (ii) of Property (A) holds. This completes the proof.
\qed
\end{proof}

\section{Selecting $\gamma_{k}$}
In this section, we present three different selection rules for the parameter $\gamma_{k}$ of the family \eqref{bbc}.

The simplest scheme for choosing $\gamma_{k}$ is to fix it for all iterations. For example, we can set $\gamma_{k}=0.1, 0.2$, etc. However, since the information carried by the two BB stepsizes changes as the iteration process going on such a fixed scheme may deteriorate the performance because it fixes the ratios of the long BB stepsize $\alpha_{k}^{BB1}$ and the short BB stepsize $\alpha_{k}^{BB2}$ contributed to the stepsize $\alpha_{k}$. Thus, it is better to vary $\gamma_{k}$ at each iteration.

One direct way for modifying $\gamma_{k}$ is, as the randomly relaxed Cauchy method \cite{raydan2002relaxed}, randomly choose it in the interval $(0,1)$. But this scheme determines the value of $\gamma_{k}$ without using any information at the current and former iterations.


The next strategy borrows the idea of cyclic gradient methods \cite{dai2003alternate,dai2005asymptotic,dai2006cyclic,raydan2002relaxed}, where a stepsize is reused for $m$ iterations. Such a cyclic scheme is superior to its noncyclic counterpart in both theory and practice. Dai and Fletcher \cite{dai2005asymptotic} showed that if the cyclic length $m$ is greater than $n/2$, the cyclic SD method is likely to be $R$-superlinearly convergent. Similar numerical convergence evidences were also observed for the CBB1 method in \cite{dai2006cyclic}. Motivated by those advantages of the cyclic scheme, for the family \eqref{bbc}, we choose $\gamma_{k}$ such that the current stepsize approximates the former one as much as possible. That is,
\begin{equation*}
  \gamma_{k}=\arg\min_{\gamma\in[0,1]}~\left|\gamma \alpha_{k}^{BB1}+(1-\gamma)\alpha_{k}^{BB2}-\alpha_{k-1}\right|,
\end{equation*}
which yields
\begin{equation}\label{eqcyc}
  \gamma_{k}^{C}=\min\left\{1,\max\left\{0,\frac{\alpha_{k-1}-\alpha_{k}^{BB2}}{\alpha_{k}^{BB1}-\alpha_{k}^{BB2}}\right\}\right\}.
\end{equation}
Clearly, $\gamma_{k}^{C}=0$ when $\alpha_{k-1}\leq\alpha_{k}^{BB2}$; $\gamma_{k}^{C}=1$ when $\alpha_{k-1}\geq\alpha_{k}^{BB1}$. This gives the following stepsize:
\begin{equation}\label{eqcyca}
  \tilde{\alpha}_{k}=\left\{
                   \begin{array}{ll}
                     \alpha_{k}^{BB2}, & \hbox{if $\alpha_{k-1}\leq\alpha_{k}^{BB2}$;} \\
                     \alpha_{k}^{BB1}, & \hbox{if $\alpha_{k-1}\geq\alpha_{k}^{BB1}$;}\\
                    \alpha_{k-1}, & \hbox{otherwise.}
                   \end{array}
                 \right.
\end{equation}
Recall that, for quadratics, $\alpha_{k}^{BB1}=\alpha_{k-1}^{SD}$ and $\alpha_{k}^{BB2}=\alpha_{k-1}^{MG}$, where
\begin{equation*}
  \alpha_{k}^{MG}=\arg\min_{\alpha\in \mathbb{R}} \|g(x_{k}-\alpha g_{k})\|=\frac{g_{k}^TAg_{k}}{g_{k}^TA^2g_{k}},
\end{equation*}
which is a short stepsize and satisfies $\alpha_{k}^{MG}\leq\alpha_{k}^{SD}$. We refer the reader to \cite{dai2003altermin} for additional details on $\alpha_{k}^{MG}$. It follows from \eqref{eqcyca} that $\tilde{\alpha}_{k}$ is adaptively selected and truncated by making use of the information of the former iteration. In particular, the stepsize is increased if the former one is too short (i.e., $\alpha_{k-1}\leq\alpha_{k-1}^{MG}$) while it is decreased if the former one is too long (i.e., $\alpha_{k-1}\geq\alpha_{k-1}^{SD}$). {{Moreover, the former stepsize will be reused if it lies in $[\alpha_{k}^{BB2},\alpha_{k}^{BB1}]$. Thus, \eqref{eqcyca} is an \emph{adaptive truncated cyclic} scheme and we will call it ATC for short.

As cyclic methods, we need to update the stepsize every $m$ iterations to avoid using a stepsize for too many iterations. Many different stepsizes can be employed. Here, we suggest three candidates. The first is the long BB stepsize $\alpha_{k}^{BB1}$, i.e.,
 \begin{equation}\label{eqcyca3}
  \alpha_{k}^{ATC1}=\left\{
                   \begin{array}{ll}
                    \alpha_{k}^{BB1}, & \hbox{if $mod(k,m)=0$;}\\
                    \tilde{\alpha}_{k}, & \hbox{otherwise.}
                   \end{array}
                 \right.
\end{equation}
The second is the short BB stepsize $\alpha_{k}^{BB2}$, i.e.,
 \begin{equation}\label{eqcyca4}
  \alpha_{k}^{ATC2}=\left\{
                   \begin{array}{ll}
                    \alpha_{k}^{BB2}, & \hbox{if $mod(k,m)=0$;}\\
                    \tilde{\alpha}_{k}, & \hbox{otherwise.}
                   \end{array}
                 \right.
\end{equation}
The last is $\alpha_{k}^P$ given by \eqref{ap}, which is a special case of the family \eqref{bbc}. That is,
\begin{equation}\label{eqcyca2}
  \alpha_{k}^{ATC3}=\left\{
                   \begin{array}{ll}
                    \alpha_{k}^P, & \hbox{if $mod(k,m)=0$;}\\
                    \tilde{\alpha}_{k}, & \hbox{otherwise.}
                   \end{array}
                 \right.
\end{equation}
In what follows we shall refer \eqref{eqcyca3}, \eqref{eqcyca4} and \eqref{eqcyca2} to as ATC1, ATC2 and ATC3,
respectively.

We tested the family \eqref{bbc} with fixed $\gamma_{k}$ on quadratic problems to see how the values of $\gamma_{k}$ affect the performance. In particular, $\gamma_{k}$ is set to 0.1, 0.3, 0.5, 0.7 and 0.9 for all $k$, respectively. The examples in \cite{dai2003altermin,friedlander1998gradient,zhou2006gradient} were employed, where $A=QVQ^T$ with
\begin{equation*}
  Q=(I-2w_3w_3^T)(I-2w_2w_2^T)(I-2w_1w_1^T),
\end{equation*}
and $w_1$, $w_2$, and $w_3$ are unitary random vectors, $V=diag(v_1,\ldots,v_n)$ is a diagonal matrix where $v_1=1$, $v_n=\kappa$, and $v_j$, $j=2,\ldots,n-1$, are randomly generated between 1 and $\kappa$. We stopped the iteration if the number of iteration exceeds 20,000 or
\begin{equation}\label{eqstop}
  \|g_k\|\leq\epsilon\|g_1\|
\end{equation}
is satisfied for some given $\epsilon$.

Four values of the condition number $\kappa$: $10^3, 10^4, 10^5, 10^6$ as well as three values of $\epsilon$: $10^{-6}, 10^{-9}, 10^{-12}$ were used. For each value of $\kappa$ and $\epsilon$, 10 instances with $v_j$ evenly distributed in $[1,\kappa]$ were generated. For each instance, the entries of $b$ were randomly generated in $[-10,10]$ and the vector $e=(1,\ldots,1)^T$ was used as the starting point.

The BB method was also run for comparison, i.e. $\gamma_{k}=1$. We compared the performance of the algorithms by the required number of iterations, as described in \cite{dolan2002}. In other words, for each method, we plot the ratio of problems for which the method is within a factor $\rho$ of the least iterations. For the 100-dimensional case, we can see from Figure \ref{randp1-fix} that the performance of the family improves as $\gamma_{k}$ becomes larger. However, for the 1000-dimensional case, Figure \ref{randp1-fix2} shows that the family \eqref{bbc} with $\gamma_{k}=0.7$ or 0.9 can outperform the BB method for some $\rho$ around 1.5. That is, for some problems, the long BB stepsize $\alpha_{k}^{BB1}$ may not be the best choice in the family.

\begin{figure}[ht!b]
\centering
    \includegraphics[width=0.75\textwidth,height=0.55\textwidth]{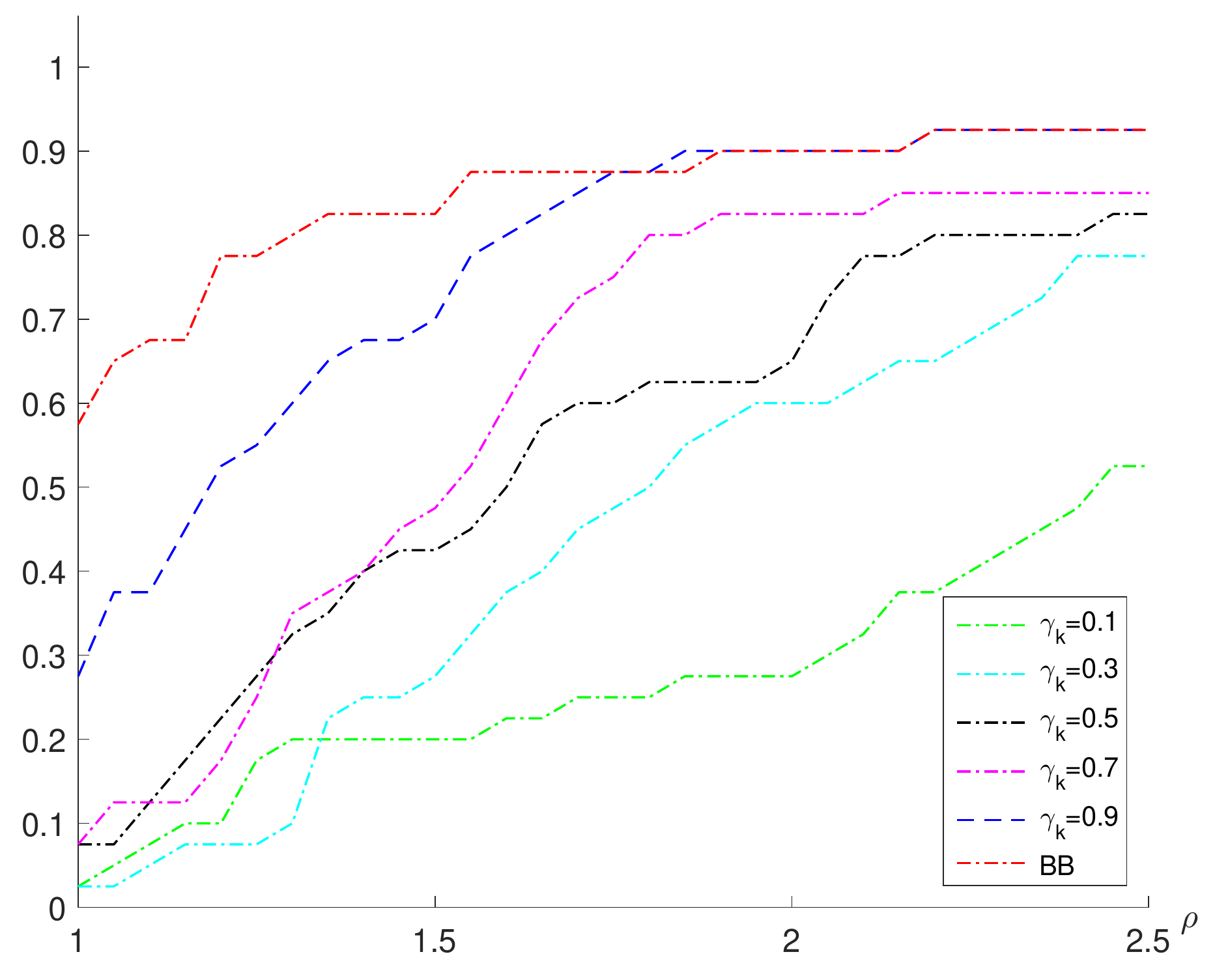}
  \caption{Performance profile of the family \eqref{bbc} with fixed $\gamma_{k}$ based on number \protect\\of iterations for 100-dimensional problems.}\label{randp1-fix}
\end{figure}

\begin{figure}[ht!b]
\centering
  \includegraphics[width=0.75\textwidth,height=0.55\textwidth]{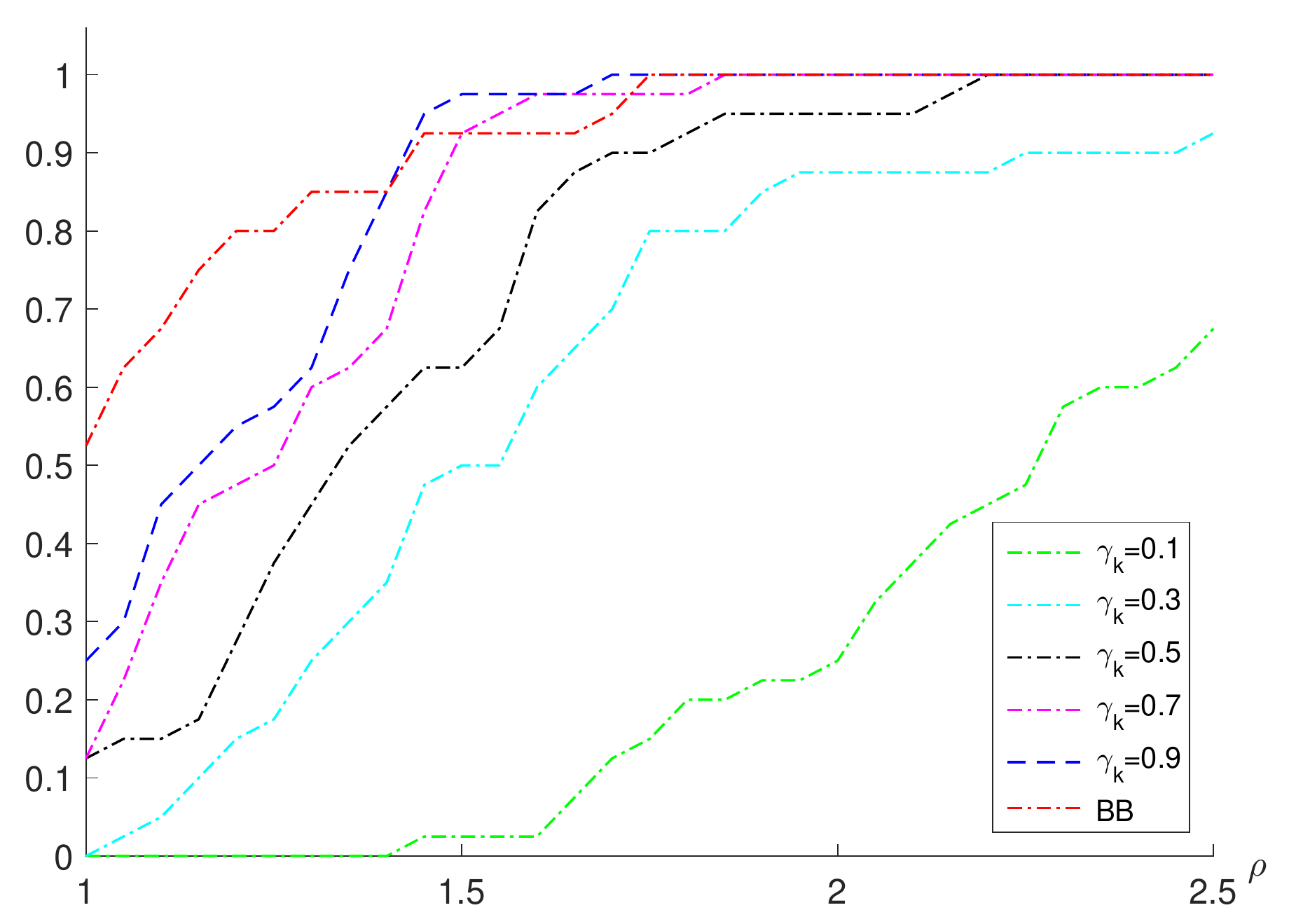}
  \caption{Performance profile of the family \eqref{bbc} with fixed $\gamma_{k}$ based on number \protect\\of iterations for 1000-dimensional problems.}\label{randp1-fix2}
\end{figure}

Then, we applied ATC1, ATC2 and ATC3 to the above problem with $n=1000$ and different $v_j$ distributions. Particularly, two sets were generated: (1) $v_j$ are evenly distributed in $[1,\kappa]$ as the former example; (2) 20\% of $v_j$ are evenly distributed in $[1,100]$ and others are in $[\frac{\kappa}{2},\kappa]$. We used three values of the condition number $\kappa=10^4, 10^5, 10^6$ and also three values of $\epsilon$ as the above problem. Other settings were same as the former example.

\begin{table}[ht!b]
\begin{center}
\caption{Number of iterations of ATC1, ATC2 and ATC3 with different $m$ on problems with different spectrum distributions.}\label{tb1}
\setlength{\leftskip}{-20pt}
\resizebox{1.1\textwidth}{45mm}{
\scriptsize
\begin{tabular}{ccccccccccccc}
 \hline
\multirow{2}{*}{Set} &\multirow{2}{*}{Method} &\multirow{2}{*}{$\kappa$} &\multicolumn{10}{c}{$m$} \\
 \cline{4-13}
& & &5 &8 &10 &20 &30 &40 &50 &60 &80 &100\\
\hline
\multirow{9}{*}{1}
&\multirow{3}{*}{ATC1}
 &$10^4$   &397.8 &435.9 &391.5 &381.6 &356.7 &350.7 &416.4 &406.4 &412.2 &407.0           \\
& &$10^5$   &3051.4 &3486.2 &2266.2 &1767.7 &1550.9 &1681.1 &1646.3 &1591.4 &1483.9 &1414.5 \\
& &$10^6$   &5864.8 &6297.0 &4467.8 &3096.0 &2608.8 &2532.9 &3154.8 &2520.5 &2853.3 &2334.8 \\

\cline{2-13}
&\multirow{3}{*}{ATC2}
 &$10^4$   &449.4 &484.2 &435.8 &407.2 &371.1 &403.9 &423.6 &384.3 &419.9 &414.8           \\
& &$10^5$   &4708.2 &3334.5 &3050.5 &2406.1 &2102.4 &2218.7 &1853.5 &1816.3 &1691.3 &1887.8 \\
& &$10^6$   &7846.2 &6356.6 &5762.5 &4562.8 &3643.1 &3587.5 &3426.5 &2852.6 &2687.7 &2862.2 \\

\cline{2-13}
&\multirow{3}{*}{ATC3}
 &$10^4$   &683.8 &620.7 &535.9 &454.4 &438.7 &444.0 &426.7 &414.3 &412.7 &427.5            \\
& &$10^5$   &9232.9 &4798.6 &8648.4 &2899.5 &2248.0 &2469.1 &2269.3 &2123.8 &1960.4 &2347.2  \\
& &$10^6$   &11526.5 &8062.5 &9560.1 &6122.4 &4625.7 &4272.6 &3786.8 &3621.4 &3304.6 &3613.7 \\

\hline
\multirow{9}{*}{2}

&\multirow{3}{*}{ATC1}
 &$10^4$    &266.8 &215.9 &272.5 &273.3 &309.6 &330.3 &357.2 &334.7 &374.4 &354.8             \\
& &$10^5$    &827.8 &794.9 &927.9 &1057.2 &1173.0 &1188.5 &1330.6 &1308.6 &1353.1 &1403.0      \\
& &$10^6$    &1322.8 &1283.3 &1553.0 &1736.2 &1875.8 &2062.4 &2165.7 &2110.3 &2239.3 &2273.7   \\

\cline{2-13}
&\multirow{3}{*}{ATC2}
 &$10^4$    &385.5 &331.1 &360.3 &350.5 &320.3 &345.2 &348.2 &336.0 &361.1 &376.8          \\
& &$10^5$    &1891.4 &1382.2 &1403.8 &1454.3 &1417.2 &1402.1 &1321.9 &1354.0 &1392.7 &1446.4\\
& &$10^6$    &3070.3 &2367.5 &2484.9 &2320.7 &2232.1 &2227.2 &2225.4 &2255.6 &2359.4 &2330.5\\

\cline{2-13}
&\multirow{3}{*}{ATC3}
 &$10^4$   &671.1 &410.7 &483.9 &371.5 &352.6 &359.7 &350.6 &359.1 &376.1 &382.6           \\
& &$10^5$   &4213.2 &1720.5 &2317.3 &1866.7 &1519.2 &1489.2 &1498.3 &1488.5 &1490.3 &1517.8 \\
& &$10^6$   &8135.7 &2895.5 &4255.8 &3046.8 &2655.6 &2534.6 &2539.4 &2393.4 &2425.7 &2552.0 \\

\hline
\end{tabular}}
\end{center}
\end{table}

We tested the three methods with different $m$. The average number of iterations are presented in Table \ref{tb1}. We can see that, for each $\kappa$ and $m$, ATC1 often outperforms ATC2 and ATC3. The performances of the three methods do not improve as $m$ increases. For the first set of problems, ATC1 with $m=30$ performs better than other values. For the second set of problems, ATC1 with $m=8$ dominates others. Thus, in the next section we only run ATC1 using these settings.

We close this section by mentioning that there are many other different rules for computing the parameter $\gamma_{k}$. For example, as the alternate gradient method \cite{dai2003alternate,dai2005projected}, we can choose $\gamma_{k}$ to alternate short stepsizes and long stepsizes. In addition, we can also use sophisticated random schemes for $\gamma_{k}$, see \cite{kalousek2017steepest} and references therein.

}}

\section{{{Numerical results}}}

In this section, we present numerical results of the family \eqref{bbc} with different settings of the parameter $\gamma_k$. We compare the performance of the ATC1 method with the following methods:
the family \eqref{bbc} with $\gamma_k$ randomly chose in $(0,1)$ (RAND), BB \cite{Barzilai1988two}, ALBB \cite{dai2005projected}, ABB \cite{zhou2006gradient}, CBB1 \cite{dai2006cyclic}, CBB2, CP, Dai-Yuan (DY) \cite{dai2005analysis}, ABBmin1 and ABBmin2 \cite{frassoldati2008new}, SDC \cite{de2014efficient}, and the family \eqref{bbc} with basic adaptive truncated cyclic scheme \eqref{eqcyca} (ATC). Since the SDC method performs better than its monotone counterpart for most problems, we only run SDC.


All methods were implemented in MATLAB (v.9.0-R2016a). All the runs were carried out on a PC with an Intel Core i7, 2.9 GHz processor and 8 GB of RAM running Windows 10 system. Moreover, we stopped the iteration if the number of iteration exceeds 20,000 or \eqref{eqstop} is satisfied for some given $\epsilon$.

Firstly, we considered randomly generated problems in the former section with different $v_j$ distributions as shown in Table \ref{tbspe}. The first two sets are same as the former section. For the third set, half of $v_j$ are in $[1,100]$ and the other half in $[\frac{\kappa}{2},\kappa]$; for the fourth set, 80\% of $v_j$ are evenly distributed in $[1,100]$ and others are in $[\frac{\kappa}{2},\kappa]$. The fifth set has 20\% of $v_j$ are in $[1,100]$, 20\% of $v_j$ are in $[100,\frac{\kappa}{2}]$ and the others in $[\frac{\kappa}{2},\kappa]$. The last two sets only has either 10 small $v_j$ or 10 large $v_j$. Other settings were also same as the former section.

For the three cyclic methods: CBB1, CBB2 and CP, the best $m$ among $\{3,4,\ldots,10\}$ was chosen. In particular, $m=3$ for CBB1 and $m=4$ for CBB2 and CP. As in \cite{frassoldati2008new}, the parameters used by the ABBmin1 and ABBmin2 methods are set to $\tau=0.8$, $m=9$ and $\tau=0.9$, respectively. While $\tau=0.1$ was used for the ABB method which is better than 0.15 and 0.2 in our test. The pair $(h,m)$ of the SDC method was set to $(8,6)$ which is better than other choices. As for our ATC1 method, we set $m=30$ for the first and fifth sets of problems and $m=8$ for other sets.




\begin{table}[ht!b]
\setlength{\tabcolsep}{0.3ex}
\begin{center}
\caption{Distributions of $v_j$.}\label{tbspe}
\begin{tabular}{|c|c|}
 \hline
 \multirow{1}{*}{Set} &\multicolumn{1}{c|}{Spectrum} \\
\hline
\multirow{1}{*}{1} &$\{v_2,\ldots,v_{n-1}\}\subset(1,\kappa)$	\\
\hline
 \multirow{2}{*}{2}
&$\{v_2,\ldots,v_{n/5}\}\subset(1,100)$	\\
&$\{v_{n/5+1},\ldots,v_{n-1}\}\subset(\frac{\kappa}{2},\kappa)$	\\
 \hline

\multirow{2}{*}{3}
&$\{v_2,\ldots,v_{n/2}\}\subset(1,100)$	\\
&$\{v_{n/2+1},\ldots,v_{n-1}\}\subset(\frac{\kappa}{2},\kappa)$	\\
 \hline
\multirow{2}{*}{4}
&$\{v_2,\ldots,v_{4n/5}\}\subset(1,100)$	\\
&$\{v_{4n/5+1},\ldots,v_{n-1}\}\subset(\frac{\kappa}{2},\kappa)$	\\
 \hline
\multirow{3}{*}{5}
&$\{v_2,\ldots,v_{n/5}\}\subset(1,100)$	\\
&$\{v_{n/5+1},\ldots,v_{4n/5}\}\subset(100,\frac{\kappa}{2})$	\\
&$\{v_{4n/5+1},\ldots,v_{n-1}\}\subset(\frac{\kappa}{2},\kappa)$	\\
 \hline
 \multirow{2}{*}{6}
&$\{v_2,\ldots,v_{10}\}\subset(1,100)$	\\
&$\{v_{11},\ldots,v_{n-1}\}\subset(\frac{\kappa}{2},\kappa)$	\\
 \hline
 \multirow{2}{*}{7}
&$\{v_2,\ldots,v_{n-10}\}\subset(1,100)$	\\
&$\{v_{n-9},\ldots,v_{n-1}\}\subset(\frac{\kappa}{2},\kappa)$	\\
 \hline
\end{tabular}
\end{center}
\end{table}

It can be observed from Table \ref{tb3} that, the rand scheme performs worse than the fixed scheme with $\gamma_k=1$, i.e., the BB method. The ATC method is competitive with the BB method and dominates the ALBB and CP methods for most test problems. The CBB2 method clearly outperforms the other two cyclic methods: CBB1 and CP. Moreover, the CBB2 method performs much better than the ATC and ABB methods except the first and fifth sets of problems. The CBB2 method is even faster than the DY method. Although the ABBmin2 method is the fastest one for solving the first and sixth sets of problems, it is worse than the ABBmin1, SDC and ATC1 methods for the other five sets of problems. Our ATC1 method is faster than the CBB2 method except the last set of problems. In addition, the ATC1 method is much better than the ABBmin1 and SDC methods on the first set of problems and comparable to them on the other sets of problems. Furthermore, for each tolerance, our ATC1 method is the fastest one in the sense of total number of iterations.

\begin{table}[h]
\begin{center}
\caption{Number of iterations of compared methods on problems in Table \ref{tbspe}.}\label{tb3}
\setlength{\leftskip}{-20pt}
\resizebox{1.2\textwidth}{45mm}{
\scriptsize
\begin{tabular}{ccccccccccccccc}
 \hline
\multirow{2}{*}{Set} &\multirow{2}{*}{$\epsilon$} &\multicolumn{13}{c}{Method} \\
 \cline{3-15}
& &RAND &BB &ALBB &ABB &CBB1 &CBB2 &CP &DY &ABBmin1 &ABBmin2 &SDC &ATC &ATC1\\
\hline
\multirow{3}{*}{1}
  &$10^{-6}$  &\multirow{1}{*}{487.8}  &\multirow{1}{*}{491.6}  &\multirow{1}{*}{521.3}   &292.6   &558.5      &538.3     &702.5   &\multirow{1}{*}{352.0} &\multirow{1}{*}{403.5} &\multirow{1}{*}{261.7} &\multirow{1}{*}{385.6} &\multirow{1}{*}{474.5}         &356.7   \\
  &$10^{-9}$  &\multirow{1}{*}{5217.4} &\multirow{1}{*}{3471.6} &\multirow{1}{*}{8502.9}  &1020.8  &4715.9     &2866.7    &9019.7  &\multirow{1}{*}{3202.1} &\multirow{1}{*}{2424.2} &\multirow{1}{*}{509.3} &\multirow{1}{*}{2265.5} &\multirow{1}{*}{2733.1}  &1550.9   \\
  &$10^{-12}$ &\multirow{1}{*}{9328.7} &\multirow{1}{*}{7241.4} &\multirow{1}{*}{10604.0} &1456.1  &8391.7     &5300.3    &11156.0 &\multirow{1}{*}{6676.7} &\multirow{1}{*}{4695.5} &\multirow{1}{*}{660.3} &\multirow{1}{*}{3786.6} &\multirow{1}{*}{4352.1} &2608.8  \\
\hline

\multirow{3}{*}{2}
 &$10^{-6}$  &\multirow{1}{*}{482.1}  &\multirow{1}{*}{433.4}  &\multirow{1}{*}{410.4}   &334.1   &322.1      &245.2     &611.9   &\multirow{1}{*}{373.1} &\multirow{1}{*}{202.5} &\multirow{1}{*}{343.1} &\multirow{1}{*}{224.4} &\multirow{1}{*}{443.3} 			     &215.9    \\
 &$10^{-9}$  &\multirow{1}{*}{2380.6} &\multirow{1}{*}{1938.6} &\multirow{1}{*}{2433.6}  &1558.7  &1629.8     &865.0     &6014.9  &\multirow{1}{*}{1565.1} &\multirow{1}{*}{782.5} &\multirow{1}{*}{1397.6} &\multirow{1}{*}{837.6} &\multirow{1}{*}{1984.0}     &794.9    \\
 &$10^{-12}$ &\multirow{1}{*}{4234.4} &\multirow{1}{*}{3211.8} &\multirow{1}{*}{3970.4}  &2629.6  &2752.4     &1465.9    &8527.3  &\multirow{1}{*}{2789.8} &\multirow{1}{*}{1215.4} &\multirow{1}{*}{2235.1} &\multirow{1}{*}{1398.1} &\multirow{1}{*}{3456.8}   &1283.3  \\

\hline

\multirow{3}{*}{3}
 &$10^{-6}$  &\multirow{1}{*}{611.7} &\multirow{1}{*}{527.9}  &\multirow{1}{*}{513.0}    &492.6   &472.1      &274.2     &900.5  &\multirow{1}{*}{460.8} &\multirow{1}{*}{250.6} &\multirow{1}{*}{509.3} &\multirow{1}{*}{301.0} &\multirow{1}{*}{559.3} 				   &272.4    \\
 &$10^{-9}$  &\multirow{1}{*}{2499.9} &\multirow{1}{*}{1918.0} &\multirow{1}{*}{2315.7}  &1723.7  &1762.1     &888.7     &5426.7 &\multirow{1}{*}{1671.8} &\multirow{1}{*}{753.1} &\multirow{1}{*}{1560.5} &\multirow{1}{*}{931.1} &\multirow{1}{*}{2109.9}     &830.7    \\
 &$10^{-12}$ &\multirow{1}{*}{4257.4} &\multirow{1}{*}{3180.6} &\multirow{1}{*}{4122.8}  &2747.9  &2904.7     &1459.4    &8540.3 &\multirow{1}{*}{2791.3} &\multirow{1}{*}{1230.7} &\multirow{1}{*}{2521.9} &\multirow{1}{*}{1440.9} &\multirow{1}{*}{3340.7}   &1398.5  \\

\hline

\multirow{3}{*}{4}
 &$10^{-6}$  &\multirow{1}{*}{790.2}  &\multirow{1}{*}{684.2}  &\multirow{1}{*}{627.8}    &599.3   &550.2      &344.6     &1085.1   &\multirow{1}{*}{560.3} &\multirow{1}{*}{312.7} &\multirow{1}{*}{653.6} &\multirow{1}{*}{355.6} &\multirow{1}{*}{714.4}      &327.8    \\
 &$10^{-9}$  &\multirow{1}{*}{2662.4} &\multirow{1}{*}{2060.8} &\multirow{1}{*}{2315.7}   &1827.5  &1777.6     &930.7     &6134.2   &\multirow{1}{*}{1814.3} &\multirow{1}{*}{871.4} &\multirow{1}{*}{1736.6} &\multirow{1}{*}{915.8} &\multirow{1}{*}{2382.2}   &886.1    \\
 &$10^{-12}$ &\multirow{1}{*}{4477.3} &\multirow{1}{*}{3135.3} &\multirow{1}{*}{4035.0}   &3024.8  &2971.4     &1515.4    &8898.2   &\multirow{1}{*}{2996.0} &\multirow{1}{*}{1311.9} &\multirow{1}{*}{2672.2} &\multirow{1}{*}{1480.0} &\multirow{1}{*}{3594.2} &1362.9  \\

\hline

\multirow{3}{*}{5}
 &$10^{-6}$   &\multirow{1}{*}{1248.1} &\multirow{1}{*}{1119.5} &\multirow{1}{*}{1361.4}   &870.2   &1247.4     &1124.9    &2049.5   &\multirow{1}{*}{841.9} &\multirow{1}{*}{948.8} &\multirow{1}{*}{1072.6} &\multirow{1}{*}{925.8} &\multirow{1}{*}{1632.5} 	   &901.0    \\
 &$10^{-9}$   &\multirow{1}{*}{6536.9} &\multirow{1}{*}{5222.7} &\multirow{1}{*}{9072.8}   &3187.7  &6045.2     &4903.1    &9313.9   &\multirow{1}{*}{4188.9} &\multirow{1}{*}{3932.2} &\multirow{1}{*}{3679.1} &\multirow{1}{*}{4149.7} &\multirow{1}{*}{6185.7}     &3339.1   \\
 &$10^{-12}$  &\multirow{1}{*}{9873.3} &\multirow{1}{*}{8451.2} &\multirow{1}{*}{10733.9}  &4944.8  &9206.7     &8306.6    &11347.6  &\multirow{1}{*}{7803.7} &\multirow{1}{*}{6305.4} &\multirow{1}{*}{5732.7} &\multirow{1}{*}{6722.3} &\multirow{1}{*}{9098.2}    &5474.6  \\
\hline

\multirow{3}{*}{6}
 &$10^{-6}$  &\multirow{1}{*}{295.1}  &\multirow{1}{*}{265.8}  &\multirow{1}{*}{212.7}   &183.1   &231.9      &141.7     &411.0    &\multirow{1}{*}{189.2} &\multirow{1}{*}{124.3} &\multirow{1}{*}{142.3} &\multirow{1}{*}{162.7} &\multirow{1}{*}{272.8} 			 &135.0    \\
 &$10^{-9}$  &\multirow{1}{*}{2089.6} &\multirow{1}{*}{1577.1} &\multirow{1}{*}{1552.5}  &1036.3  &1449.0     &644.5     &5093.5   &\multirow{1}{*}{1245.2} &\multirow{1}{*}{637.9} &\multirow{1}{*}{637.4} &\multirow{1}{*}{704.3} &\multirow{1}{*}{1558.7}   &604.7    \\
 &$10^{-12}$ &\multirow{1}{*}{3888.4} &\multirow{1}{*}{2516.6} &\multirow{1}{*}{3159.2}  &1659.9  &2474.0     &1117.4    &8513.8   &\multirow{1}{*}{2341.5} &\multirow{1}{*}{1048.3} &\multirow{1}{*}{914.3} &\multirow{1}{*}{1233.4} &\multirow{1}{*}{2662.1} &1004.0  \\
 \hline

\multirow{3}{*}{7}
 &$10^{-6}$  &\multirow{1}{*}{1058.1} &\multirow{1}{*}{933.5}  &\multirow{1}{*}{1048.5}   &868.5   &610.4      &332.8     &834.1   &\multirow{1}{*}{796.4} &\multirow{1}{*}{392.3} &\multirow{1}{*}{797.3} &\multirow{1}{*}{468.1} &\multirow{1}{*}{818.6} 	 &418.7     \\
 &$10^{-9}$  &\multirow{1}{*}{2641.2} &\multirow{1}{*}{2216.7} &\multirow{1}{*}{2393.2}   &2116.2  &1308.4     &713.6     &2040.6  &\multirow{1}{*}{1896.7} &\multirow{1}{*}{837.6} &\multirow{1}{*}{1707.7} &\multirow{1}{*}{962.6} &\multirow{1}{*}{1772.2}   &934.7     \\
 &$10^{-12}$ &\multirow{1}{*}{4166.0} &\multirow{1}{*}{3210.6} &\multirow{1}{*}{3774.9}   &3094.9  &2043.1     &1075.0    &3241.7  &\multirow{1}{*}{2945.7} &\multirow{1}{*}{1253.0} &\multirow{1}{*}{2467.6} &\multirow{1}{*}{1397.0} &\multirow{1}{*}{2652.3} &1354.3   \\
\hline
\multirow{3}{*}{Total}
 &$10^{-6}$   &4973.1 &4455.9 &4695.1 &3640.4 &3992.6 &3001.7 &6594.6 &3573.7 &2634.7 &3779.9 &2823.2 &4915.4 &2627.5  \\
 &$10^{-9}$   &24028.0 &18405.5 &28586.4 &12470.9 &18688.0 &11812.3 &43043.5 &15584.1 &10238.9 &11228.2 &10766.6 &18725.8 &8941.1  \\
 &$10^{-12}$  &40225.5 &30947.5 &40400.2 &19558.0 &30744.0 &20240.0 &60224.9 &28344.7 &17060.2 &17204.1 &17458.3 &29156.4 &14486.4 \\
\hline
\end{tabular}}
\end{center}
\end{table}

Then, the compared methods were applied to the non-rand quadratic minimization problem in \cite{de2014efficient} which is more difficult than its rand counterpart. In particular, $A$ is diagonal whose elements are given by
\begin{equation}\label{pro2}
  A_{jj}=\left\{
     \begin{array}{ll}
       1, & \hbox{$j=1$,} \\
       10^{\frac{ncond}{n-1}(n-j)}, & \hbox{$j=2,\ldots,n-1$,} \\
       \kappa, & \hbox{$j=n$,}
     \end{array}
   \right.
\end{equation}
where $ncond=\log_{10} \kappa$, and the vector $b$ is null. We tested 10,000-dimensional problems with three different condition numbers: $\kappa=10^4, 10^5, 10^6$. The stopping condition \eqref{eqstop} was employed with $\epsilon=10^{-6}, 10^{-9}, 10^{-12}$ for all methods. For each value of $\kappa$ or $\epsilon$, 10 different starting points with entries in $[-10,10]$ were randomly generated.

Due to the performances of these compared methods on the above problems, only fast methods were tested, i.e., ABB, CBB2, DY, ABBmin1, ABBmin2, SDC, ATC and ATC1. The numbers of iterations averaged over those starting points of each method are listed in Table \ref{tb4}. Here, the results of the SDC method were obtained with the best choice of parameter setting in \cite{de2014efficient}, i.e., $h=30$ and $m=2$.

From Table \ref{tb4} we can see that the ATC method is slightly better than the CBB2 method and comparable to the DY method for most problems. The ABB method performs better than the ABBmin1 method. Although the ABB method is slower than the ABBmin2 method when the tolerance is low, it wins if a tight tolerance is required. Our ATC1 method is competitive with other methods especially for a tight tolerance. Moreover, our ATC1 method takes least total number of iterations to meet the required tolerance.

\begin{table}[ht!b]
\begin{center}
\caption{Number of iterations of compared methods on non-rand quadratic problems.}\label{tb4}
\setlength{\leftskip}{-12pt}
\scriptsize
\begin{tabular}{cccccccccc}
 \hline
\multirow{2}{*}{$\epsilon$} &\multirow{2}{*}{$\kappa$} &\multicolumn{8}{c}{Method} \\
 \cline{3-10}
& &ABB  &CBB2 &DY &ABBmin1 &ABBmin2 &SDC &ATC &ATC1\\
\hline
\multirow{3}{*}{$10^{-6}$}
&$10^4$   &531.3     &670.7        &\multirow{1}{*}{517.4} &\multirow{1}{*}{531.1} &\multirow{1}{*}{510.4} &\multirow{1}{*}{547.3} &\multirow{1}{*}{574.4}              &558.8  \\
&$10^5$   &938.1     &1167.5       &\multirow{1}{*}{996.0} &\multirow{1}{*}{974.6} &\multirow{1}{*}{872.0} &\multirow{1}{*}{990.6} &\multirow{1}{*}{1097.2}             &1011.6 \\
&$10^6$   &1368.4    &1550.6       &\multirow{1}{*}{1452.4} &\multirow{1}{*}{1362.9} &\multirow{1}{*}{1299.4} &\multirow{1}{*}{1430.8} &\multirow{1}{*}{1511.0}         &1408.7 \\
\hline

\multirow{3}{*}{$10^{-9}$}
&$10^4$    &1235.2    &1581.8      &\multirow{1}{*}{1243.7} &\multirow{1}{*}{1277.2} &\multirow{1}{*}{1284.5} &\multirow{1}{*}{1203.7} &\multirow{1}{*}{1340.2}         &1289.9 \\
&$10^5$    &2593.3    &3443.5      &\multirow{1}{*}{3081.4} &\multirow{1}{*}{2846.5} &\multirow{1}{*}{2826.8} &\multirow{1}{*}{2694.9} &\multirow{1}{*}{3155.9}         &2719.7 \\
&$10^6$    &3873.3    &4718.3      &\multirow{1}{*}{4807.9} &\multirow{1}{*}{4192.1} &\multirow{1}{*}{3869.5} &\multirow{1}{*}{4102.6} &\multirow{1}{*}{4439.5}         &4004.9 \\
\hline

\multirow{3}{*}{$10^{-12}$}
&$10^4$   &2365.6    &2839.0      &\multirow{1}{*}{2132.7} &\multirow{1}{*}{2083.7} &\multirow{1}{*}{3042.0} &\multirow{1}{*}{1891.7} &\multirow{1}{*}{2546.6}         &2092.9 \\
&$10^5$   &7331.1    &9015.1      &\multirow{1}{*}{9696.2} &\multirow{1}{*}{7916.3} &\multirow{1}{*}{7491.7} &\multirow{1}{*}{7304.3} &\multirow{1}{*}{8844.2}         &7238.1 \\
&$10^6$   &12124.9   &14050.5     &\multirow{1}{*}{17329.3} &\multirow{1}{*}{12792.4} &\multirow{1}{*}{11653.1} &\multirow{1}{*}{11775.5} &\multirow{1}{*}{15157.7}    &11600.6\\
\hline

\multicolumn{1}{c}{Total} &
 &32361.2 &39037.0 &41257.0 &33976.8 &32849.4 &31941.4 &38666.7 &31925.2      \\

\hline
\end{tabular}
\end{center}
\end{table}

\section{Conclusions}
We have proposed a family of spectral gradient methods which calculates the stepsize by the convex combination of the long BB stepsize and the short BB stepsize, i.e., $\alpha_k=\gamma_k\alpha_k^{BB1}+(1-\gamma_k)\alpha_k^{BB2}$. Similar to the two BB stepsizes, each stepsize in the family possesses certain quasi-Newton property. In addition, $R$-superlinear and $R$-linear convergence of the family were established for two-dimensional and $n$-dimensional strictly convex quadratics, respectively. Furthermore, with different choices of the parameter $\gamma_k$, we obtain different stepsizes and gradient methods. The family provides us an alternative for the stepsize of gradient methods which can be easily extended to general problems by incorporating some line searches.

Since the parameter $\gamma_k$ affects the value of $\alpha_k$ and hence the efficiency of the method, it is interesting to investigate how to choose a proper $\gamma_k$ to achieve satisfactory performance. We have proposed and tested three different selection rules for $\gamma_k$, among which {{the adaptive truncated cyclic scheme with the long BB stepsize $\alpha_k^{BB1}$, i.e., the ATC1 method performs best. In addition, our ATC1 method is comparable to the state-of-the-art gradient methods including the Dai-Yuan, ABBmin1, ABBmin2 and SDC methods. One interesting question is how to design an efficient scheme to choose $m$ adaptively for the proposed method. This will be our future work.
}}



\end{document}